\documentclass{svjour3}

\smartqed  
\usepackage{graphicx}
\usepackage{amsmath}
\usepackage{latexsym}
\usepackage{amssymb}
\usepackage{mathptmx}      

\usepackage[english]{babel}

\font\gothic=eufm10 scaled 1100

\def\CC#1{\mathrm{#1}}                 
\def\GG#1{\hbox{\gothic{#1}}}        
\def\RR{\mathbb{R}}

\def\D{\,{\rm d}}
\newcommand{\fix}{\hbox{{\rm fix}}}
\newcommand{\tr}{\hbox{{\rm tr}}}
\newcommand{\gl}{\hbox{{\gothic{gl}}}}
\newcommand{\so}{\hbox{{\gothic{so}}}}
\newcommand{\g}{\mbox{\boldmath$g$\unboldmath}}
\newcommand{\dexp}{\hbox{{\rm dexp}}}

\newcommand{\refp}[1]{(\ref{#1})}
\newcommand{\ad}{\hbox{{\rm ad}}}

\newcommand{\id}{\hbox{{\rm id}}}
\newcommand{\ghe}{\hbox{{\gothic{g}}}}
\newcommand{\GL}{{\rm GL}}
\def\Frac#1#2{{\textstyle \frac{#1}{#2}}}
\def\OO#1{{\cal O}\!\left(#1\right)}

\newenvironment{meqn}
{\arraycolsep=1.4pt
  
  \begin{array}{rcl}}
  {\end{array}}

\journalname{\ }
  
\begin{document}

\title{
Symmetric spaces and Lie triple systems in numerical analysis of differential equations
} 

\titlerunning{Symmetrics spaces and Lie triple systems in numerical analysis of DE}
  
\author{
H. Z. Munthe-Kaas, \and
 G. R. W. Quispel, \and
A. Zanna
}

\institute{
H. Z. Munthe-Kaas, A. Zanna  \at
              Matematisk institutt, Universitet i Bergen, Johannes Brunsgt 12, N-5008, Bergen, Norway \\
              \email{Hans.Munthe-Kaas@math.uib.no, Antonella.Zanna@math.uib.no}
              \and
	G. R. W. Quispel \at Department of Mathematics, La Trobe
  University, Bundoora,  Melbourne 3083, Australia. \\
  	\email{r.quispel@latrobe.edu.au}
}

\date{\today}

\maketitle
\begin{abstract}
  A remarkable number of different numerical
  algorithms can be understood and analyzed using the concepts of
  symmetric spaces and Lie triple systems, which are well known
  in differential geometry from the study of spaces of constant
  curvature and their tangents.
  This theory can be used to unify a range of different
  topics, such as polar-type matrix decompositions,
  splitting methods for computation of the matrix exponential,
  composition  of selfadjoint numerical integrators and dynamical systems 
  with symmetries and reversing symmetries.
  The thread of this paper is the following: involutive automorphisms on groups induce a factorization at a group level, and a splitting at the algebra level. 
  In this paper we will give an introduction to the
  mathematical theory behind these constructions, and review recent
  results. 
  Furthermore, we present a new Yoshida-like technique, for self-adjoint numerical schemes, that allows to increase the order 
  of preservation of symmetries by two units.
  Since all the time-steps are positive, the technique is particularly 
  suited to stiff problems, where a negative time-step can cause 
  instabilities.


\end{abstract}

\section{Introduction}
\label{sec:1}
In numerical analysis there exist numerous examples of objects forming a
\emph{group}, i.e.\ objects that compose in an associative manner,
have an inverse and identity element.
Typical examples are the group of orthogonal matrices or the group  of
Runge--Kutta methods. \emph{Semigroups}, sets of objects close under composition
but not inversion, like for instance the set of all matrices and
explicit Runge--Kutta methods,  are also well studied in literature.

However, there are important examples of objects that are neither a group nor a
semigroup. One important case is the class of objects closed under a 
`sandwich-type' product, $(a,b) \mapsto aba$. For example, the collection of all
symmetric positive definite matrices and all selfadjoint Runge--Kutta
methods. 
The sandwich-type composition $aba$ for numerical integrators was studied
at length in~\cite{mclachlan98nit} and references therein. However, if
inverses are well defined, we may wish to replace
the sandwich product with the algebraically nicer \emph{symmetric
  product}
$(a,b) \mapsto ab^{-1}a$. Spaces closed under such products are called 
symmetric spaces and are the objects of study in this paper. 
There is a parallel between  the theory of Lie groups and that of 
symmetric spaces. For Lie groups, 
fundamental  tools are the Lie algebra (tangent space at the
identity, closed under commutation) and the exponential map from the Lie 
algebra to the Lie group.
 In the theory of symmetric spaces there is a similar notion of tangent
space. The resulting object is called a Lie triple system (LTS), and is
closed under \emph{double commutators}, $[X,[Y,Z]]$.  Also 
in this case, the exponential maps the LTS into the symmetric space.

An important decomposition theorem is associated with symmetric spaces and Lie triple systems: 
Lie algebras can be decomposed into a direct sum of a LTS and a subalgebra. The well known splitting of a
matrix as a sum of a symmetric and a skew symmetric matrix is an
example of such a decomposition, the skew symmetric matrices are
closed under commutators, while the symmetric matrices are closed
under double commutators. 
Similarly, at the group level, there are decompositions of Lie
groups into a product of a symmetric space and a Lie subgroup. The
matrix polar decomposition, where a matrix is written as the product
of a symmetric positive definite matrix and an orthogonal matrix is
one example. 

In this paper, we are concerned with the application of such structures to the numerical analysis of differential equations of evolution.
The paper is organised as follows: in \S2 we review some general theory of symmetric spaces and Lie triple systems.
Applications of this theory in numerical analysis of differential equations are discussed in 
\S3, which, in turn, can be divided into two parts. In the first (\S3.1--\S3.3), we review and discuss the case of differential equations on matrix groups. The properties of these decompositions and numerical algorithms based
on them are studied in a series of 
papers.  Polar-type decompositions are studied in detail 
in~\cite{munthe-kaas01gpd},  with special emphasis on optimal 
approximation results. The paper~\cite{zanna00rrf} is concerned with important
recurrence relations for polar-type decompositions, similar to the
Baker-Campbell-Hausdorff formula for Lie groups, while 
\cite{zanna01gpd,iserles05eco,krogstad01alc} employ this theory to reduce the 
implementation costs of numerical methods on Lie groups 
\cite{iserles00lgm,munthe-kaas99hor}.
We mention that polar-type decompositions are also 
closely related to the more special root-space decomposition employed 
in numerical integrators for differential equations on Lie groups
in~\cite{owren01imb}. In~\cite{krogstad03gpc} it is shown that
the generalized polar decompositions can be employed in
cases where the theory of~\cite{owren01imb} cannot be used.

In the second part of \S3 (\S\ref{sec:3.3} and beyond), we will consider the application of this theory to numerical methods for the solution of differential equations with symmetries and reversing symmetries. By backward error analysis, numerical methods can be thought of as exact flows of nearby vector fields. The main goal is then to remove from the numerical method the \emph{undesired} part of the error (either the one destroying the symmetry or the reversing symmetry). These error terms in the modified vector field generally consist of complicated derivatives of the vector field itself and are not explicitly calculated, they are just used formally for analysis of the method. In this context, the main tools are compositions at the group level, using the flow of the numerical method and the symmetries/reversing symmetries, together with their inverses. There is a substantial difference between preserving reversing symmetries and symmetries for a numerical method: the first can be always be attained by a finite (2 steps) composition (Scovel projection \cite{scovel91sni}), the second requires in general an infinite composition. Thus symmetries are generally more difficult to retain than reversing symmetries. For the retention of symmetry, we review the Thue--Morse symmetrization technique for arbitrary methods and present a new Yoshida-like symmetry retention technique for self-adjoint methods. The latter has always positive intermediate step sizes and can be of interest in the context of stiff problems, which typically require step size restrictions. We illustrate the use of these symmetrisation methods by some numerical experiments.

Finally, Section~4 is devoted to some concluding remarks.

\section{General theory of symmetric spaces and Lie triple systems}
\label{sec:2}
In this section we present some background theory for symmetric spaces
and Lie triple systems. We expect the reader to be familiar with some
basic concepts of differential geometry, like manifolds, vector
fields, etc.  
For a more detailed treatment of symmetric spaces we refer the reader
to \cite{helgason78dgl} and \cite{loos69sp1} which also
constitute the main reference of the material presented in this section. 

We shall also follow (unless otherwise mentioned) the notational
convention of \cite{helgason78dgl}: in particular, $M$ is a set
(manifold), the letter $G$ is reserved for groups and Lie groups,
gothic letters denote Lie algebras and Lie triple systems,
latin lowercase letters denote
Lie-group elements and latin uppercase letters denote Lie-algebra
elements. The identity element of a group will usually be denoted by
$e$ and the identity mapping by $\id$.

\subsection{Symmetric spaces}
\label{sec:2.1}
\begin{definition}[See \cite{loos69sp1}]\label{def:2.1}
  A {\em symmetric space\/} is a manifold $M$ with a 
     differentiable \emph{symmetric product} $\cdot$ obeying the following 
     conditions:
    \begin{itemize}
        \item[(i)] $x \cdot x  =  x,$
        \item[(ii)] $x \cdot(x \cdot y) =  y,$
        \item[(iii)] $x \cdot(y \cdot z)  =  (x \cdot y) \cdot (x \cdot z),$
        \end{itemize}
        and moreover
        \begin{itemize}
        \item[(iv)] every $x$ has a neighbourhood $U$ such that  for all $y$ in $U$ $x 
        \cdot y = y$ implies $y=x$. 
    \end{itemize} 
\end{definition}
The latter condition is relevant in the case of manifolds $M$ with
open set topology (as in the case of Lie groups) and can be
disregarded for sets $M$ with discrete topology: a discrete set $M$
endowed with a multiplication obeying (i)--(iii) will be also called a
symmetric space.

A {\em pointed symmetric space\/} is a pair $(M,o)$ consisting of
a symmetric space $M$ and a point $o$ called {\em base point.} Note
that when $M$ is a Lie group, it is usual to set $o=e$. Moreover, if the group is a \emph{matrix} group with the usual matrix multiplication, it is usual to set $e=I$ (identity matrix).

The left multiplication with an element $x \in M$ is denoted by $S_x$,
\begin{equation}
  S_x y = x \cdot y, \quad \forall y \in M,
  \label{eq:symmetry}
\end{equation}
and is called {\em symmetry around} $x$. Note that $S_x x = x$ because
of (i), hence $x$ is fixed point of $S_x$ and it is isolated because
of (iv). Furthermore, (ii) and (iii) imply $S_x$ is an involutive
automorphism of $M$, i.e.\ $S_x^2 = \id$.

Symmetric spaces can be constructed in several different ways, the following are important examples:
\begin{enumerate}
\item\label{item:1} Manifolds with an intrinsically defined symmetric product. As an
  example, consider the $n$-sphere as the set of
  unit vectors in $\RR^{n+1}$.  The product
  \begin{displaymath}x\cdot y = S_x y = (2xx^T -I)y
    \end{displaymath}
    turns this into a symmetric space. The above operation is the reflection of points on a sphere. This can be generalized to $m$-dimensional subspaces  in $\RR^n$ ($m\leq n$) in the following manner: Assume that $x=[x_1, \ldots, x_m]$ is a full rank matrix. Define $P_x$ the orthogonal projection operator onto the range of $x$, $P_x = x(x^*x)^{-1} x^*$. Consider the reflection $R_x= 2P_x - I$. Define $x\cdot y = R_x y$. This operation obeys the conditions (i)--(vi) whenever $x,y,z$ are $m\times n$ full rank matrices.  
In particular, note that (i) is equivalent to $R_x^2=I$, i.e.\ the reflection is an involutive matrix. 
        
\item\label{item:2} Subsets of a continuous (or discrete) group $G$ that are closed 
  under the composition $x\cdot y = x y^{-1} x$, where $xy$ is the 
  usual multiplication in $G$. Groups themselves, continuous, as in the 
  case of Lie groups, or discrete, are thus particular instances of
  symmetric spaces. As another example, consider
  the set of all symmetric positive
  definite matrices as a subset of all nonsingular matrices, which 
  forms a symmetric space with the product
 \begin{displaymath}
 a\cdot b = ab^{-1}a .
\end{displaymath} 

\item\label{item:3} Symmetric elements of automorphisms on a group. 
  An automorphism on
  a group $G$ is a map $\sigma:{G} \to {G}$ satisfying
  $\sigma(ab)=\sigma(a)\sigma(b)$. The \emph{symmetric elements} are defined as
  \begin{displaymath}
   \mathcal{A} = \{ g\in G \; : \; \sigma(g)=g^{-1} \} .
  \end{displaymath}
 It is easily verified that $\mathcal{A}$ obeys (i)--(iv) when endowed with the multiplication
  $x\cdot y = x y^{-1} x$, hence it is a symmetric space. As an
  example, symmetric matrices are symmetric elements under the matrix
  automorphism $\sigma(a) = a^{-T}$.

\item\label{item:4} Homogeneous manifolds. Given a Lie group $G$ and a subgroup $H$,
  a homogeneous manifold $M=G/H$ is the set of all left cosets of $H$ in
  $G$.  Not every homogeneous manifold possesses a product turning it into a
  symmetric space, however, we will see in Theorem~\ref{th:2.1} that any connected symmetric
  space arises in a natural manner as a homogeneous manifold.
\item\label{item:5} Jordan algebras. Let $\GG{a}$ be a finite-dimensional vector space with a bilinear multiplication\footnote{Typically non-associative.} $x\ast y$ such that 
\begin{displaymath}
	x\ast y = y\ast x, \qquad x\ast(x^2 \ast y) = x^2 \ast(x\ast y)
\end{displaymath}
(powers defined in the usual way, $x^m = x\ast x^{m-1}$), with unit element $e$. Define $L_x(y)= x\ast y$ and set $P_x = 2L_x^2 - L_{x^2}$. Then, the set of \emph{invertible} elements of $\GG {a}$ is a symmetric space with the product 
\begin{displaymath}
	x\cdot y = P_x (y^{-1}).
\end{displaymath} 
In the context of symmetric matrices, take $x\ast y = \frac12(xy+yx)$, where $xy$ denotes the usual matrix multiplication. After some algebraic manipulations, one can verify that the product $x\cdot y =P_x(y^{-1})= 2 x\ast(x\ast y^{-1}) - (x\ast x)\ast y^{-1}= xy^{-1}x$  as in example~\ref{item:2}.
\end{enumerate}

Let $G$ be a connected Lie group and let $\sigma$ be an analytic involutive
automorphism, i.e.\ $\sigma \not= \id$ and $\sigma^2 =
  \id$. Let $G^\sigma$ denote $\fix \sigma=\{ g\in G \; : \; \sigma(g)=g\}$, $G^\sigma_e$ its
connected component including the base point, in this case the
identity element $e$ and finally let $K$ be a closed subgroup such that
$G^\sigma_e \subset K \subset G^\sigma$. Set $G_\sigma = \{ x \in G:
\sigma(x) = x^{-1} \}$.

\begin{theorem}[\cite{loos69sp1}]
  \label{th:2.1}
  The homogeneous space $M = G/K$ is a symmetric space with the product $xK \cdot
  yK = x \sigma(x)^{-1} \sigma(y)K$ and $G_\sigma$ is a symmetric
  space with the product $x \cdot y = x y^{-1} x$. The space of symmetric elements $G_\sigma$
  is isomorphic to the homogeneous space $G/G^\sigma$.
   Moreover, every connected symmetric space is of the type $G/K$  and also of the type $G_\sigma$.
\end{theorem}

The interesting consequence of the above theorem is that \emph{every
connected symmetric space} is also a homogeneous space, which implies a factorization: as coset representatives for  $G/G^\sigma$ one may choose
elements of $G_\sigma$, thus any $x\in G$ can be decomposed in a
product $ x = pk$, where $ p\in G_\sigma \mbox{ and } k\in G^\sigma$.  In other words,
\begin{equation}
  \label{eq:polardecomp}
  x = pk,  \qquad \sigma(k) = k, \quad \sigma(p) = p^{-1}, \qquad \hbox{(group factorization)}  .
\end{equation}
The matrix polar decomposition is a particular example, discussed in 
\S\ref{sec:3.1}.  

The automorphism $\sigma$ on $G$ induces an automorphism on the Lie
algebra $\ghe$ and also a canonical decomposition of $\ghe$.
Let $\GG{g}$ and $\GG{k}$ denote the Lie algebras of $G$ and $K$
respectively and denote by $\D\sigma$ the differential of $\sigma$ at
$e$,
\begin{equation}
  \label{eq:2.1}
  \D\sigma(X) = \frac{\D}{\D t}\Big|_{t=0} \sigma(\exp(tX)),
  \quad \forall X \in \ghe.
\end{equation}
Note that $\D\sigma$ is an involutive automorphism of
$\ghe$ and has eigenvalues $\pm1$. Moreover, $X \in \GG{k}$ implies
$\D\sigma(X) =X$. Set $\GG{p} = \{ X \in \ghe : \D\sigma(X)
=-X\}$. Then,
\begin{equation}
  \label{eq:2.2}
  \ghe = \GG{p} \oplus \GG{k},
\end{equation}
\cite{helgason78dgl}.
 It is easily verified that
\begin{equation}
  \label{eq:2.3}
    [\GG{k}, \GG{k}] \subset \GG{k}, \quad     {[\GG{k}, \GG{p}]} \subset \GG{p}, \quad
    {[\GG{p}, \GG{p}]} \subset \GG{k},
\end{equation}
that is, $\GG{k}$ is a subalgebra of $\ghe$ while
$\GG{p}$ is an ideal in $\GG{k}$. 
Given $X \in \ghe$, its canonical decomposition $\GG{p}\oplus\GG{k}$ is $X=P+K$, with $P \in \GG{p}$ and $K \in \GG{k}$,
\begin{equation}
  \label{eq:algebrasplitting}
  X = P + K, \qquad \D\sigma(K) = K, \quad \D \sigma(P) = -P,  \qquad \hbox{(algebra splitting)}.
\end{equation}

We have already observed that there is close connection between projection matrices, reflections (involutive matrices) and hence symmetric spaces. In a linear algebra context, this statement can be formalized as follows. Recall that a matrix $\Pi$ is a \emph{projection} if $\Pi^2=\Pi$.  
\begin{lemma}
	\label{th:proj}
	To any projection matrix $\Pi$ there corresponds an involutive matrix $S= I-2\Pi$. Conversely, to any involutive matrix $S$ there 	correspond two projection matrices $\Pi_S^-= \frac12 (I-S)$ and $\Pi_S^+=\frac12(I+S)$. These projections satisfy $\Pi_S^- + 	\Pi_S^+ = I$ and $\Pi_S^- \Pi_S^+ = \Pi_S^+ \Pi_S^{-} =0 $, moreover $S\Pi_S^{\pm} = 	\pm\Pi_S$, i.e. \ the projection $\Pi_S^\pm$ projects onto the $\pm1$ eigenspace of $S$.
 \end{lemma}
Note that if $S$ is involutive, so is $-S$, which corresponds to the opposite identification of the $\pm1$ eigenspaces. A matrix $K$, whose columns are in the span of the $+1$ eigenspace is said to be \emph{block-diagonal} with respect to the automorphism, while a matrix $P$, whose columns are in the span of the $-1$ eigenspace, is said to be \emph{2-cyclic}.

In the context of Lemma~\ref{th:proj}, we recognize that $P \in \GG{p}$ is the 2-cyclic part and $K \in \GG{k} $ is the block-diagonal part. Namely, if $X$ is represented by a matrix, then
\begin{displaymath}
	X = \left( \begin{array}{cc} X^{--} & X^{-+} \\ X^{+-} & X^{++}
	\end{array} 
	\right),
\end{displaymath}
where $X^{ij} = \Pi_S^{i} X \Pi_S^j$ restricted to the appropriate subspaces. Then,  $X^{--}$ and  $X^{++}$ corresponds to the block-diagonal part $K$ and $X^{-+}, X^{+-}$ corresponding to the 2-cyclic part $P$.

In passing, we mention that the decomposition
\refp{eq:2.3} is called a {\em Cartan decomposition} whenever the 
Cartan--Killing form $B(X,Y) = \tr(\ad_{X}\ad_{Y})$ is nondegenerate, 
hence it can be used to introduce a positive bilinear form 
$B_{\D\sigma}=-B(X, \D\sigma(Y))$. If this is the case, the linear subspaces $\GG{k}, \GG{p}$ are orthogonal.

The involutive automorphism $\sigma$ need not be defined at the group level $G$ and thereafter lifted to the algebra by \refp{eq:2.1}. It is possible to proceed the other way around: an involutive algebra automorphism $\D\sigma$ on $\ghe$, which automatically produces a decomposition \refp{eq:2.2}-\refp{eq:2.3}, can be used to induce a group automorphism $\sigma$ by the relation
\begin{equation}
	\sigma (x) = \exp(\D\sigma (\log x)),
	\label{eq:downstairs}
\end{equation}
and a corresponding group factorization \refp{eq:polardecomp}.
Thus, we have an ``upstairs-downstairs'' viewpoint: the group involutive automorphisms generate corresponding algebra automorphisms and vice versa.
This ``upstairs-downstairs'' view is useful: in some cases, the group factorization \refp{eq:polardecomp} is difficult to compute starting from $x$ and $\sigma$, while the splitting at the algebra level might be easy to compute from $X$ and $\D\sigma$. In other cases, it might be the other way around.

\subsection{Lie triple systems}
In Lie group theory Lie algebras are important since they describe
infinitesimally the structure of the tangent space at the
identity. Similarly, Lie triple systems give the structure of the
tangent space of a symmetric space.
\label{sec:2.2}
\begin{definition}(\cite{loos69sp1})
\label{def:2.2}
  A vector space with a trilinear composition $[X,Y,Z]$ is called a
  {\em Lie triple system\/} (LTS) if the following identities are satisfied:
  \begin{itemize}
  \item[(i)] $[X,X,X]=0$,
  \item[(ii)] $[X,Y,Z]+[Y,Z,X] + [Z,X,Y] =0$,
  \item[(iii)] $[X,Y,[U,V,W]] = [[X,Y,U], V,W] + [U, [X,Y,V],W] +
    [U,V,[X,Y,W]]$. 
  \end{itemize}
\end{definition}
A typical way to construct a LTS is by means of an involutive
automorphism of a Lie algebra $\ghe$. With the same notation as above,
the set $\GG{p}$ is a LTS with the composition
\begin{displaymath}
  [X,Y,Z] = [[X,Y], Z].
\end{displaymath}
Vice versa, for every LTS there exists a Lie algebra $\GG{G}$ and an involutive
automorphism $\sigma$ such that the given LTS corresponds to
$\GG{p}$. The algebra $\GG{G}$ is called {\em standard
  embedding\/} of the LTS.
In general, any subset of $\ghe$ that is closed under the operator
\begin{equation}
    \CC{T}_{X} (\mbox{}\cdot\mbox{})= \ad_{X}^{2} (\mbox{}\cdot\mbox{})= [X, [X, \mbox{}\cdot\mbox{}]]
    \label{eq:T_X}
\end{equation}
is a Lie triple system. It can be shown that being closed under $\CC{T}_{X}$ guarantees 
being closed under the triple commutator. 
%

\section{Application of symmetric spaces in numerical analysis}
\label{sec:3}

\subsection{The classical polar decomposition of matrices}
\label{sec:3.1}
Let $\GL(N)$ be the group of $N\times N$ invertible 
matrices. Consider the map
\begin{equation}
  \label{eq:3.1}
  \sigma(x) = x^{-\CC{T}}, \qquad x \in \GL(N).
\end{equation}
It is clear that $\sigma$ is an involutive automorphism of
$\GL(N)$. Then, according to Theorem~\ref{th:2.1}, the set of
symmetric elements $G_\sigma = \{ x \in \GL(N): \sigma(x) =x^{-1}\}$
is a symmetric space. We observe that $G_\sigma$ is the set of
invertible symmetric matrices. The symmetric space $G_\sigma$ is
disconnected and particular mention deserves its connected component
containing the identity matrix $I$, since it reduces to the set of
symmetric positive definite matrices. The subgroup $G^\sigma$ consists of
all orthogonal matrices. The decomposition~(\ref{eq:polardecomp}) is
the classical polar decomposition, any nonsingular matrix can be written as a
product of a symmetric matrix and an orthogonal matrix. If we restrict
the symmetric matrix to the symmetric positive definite (spd) matrices, then
the decomposition is unique. 
In standard notation, $p$ is denoted by $s$ (spd matrix), while $k$ is denoted by $q$ (orthogonal matrix).\footnote{Usually, 
matrices are denoted by upper case letters. Here we hold on the convention described in \S.2.} 
At the algebra level, the corresponding splitting is $\ghe = \GG{p}\oplus \GG{k}$, where 
$ 
    \GG{k} = \{ X \in \gl(N) : \D\sigma(X) = X \} = \so(N),
$ 
is the classical algebra of skew-symmetric matrices, while
$ 
    \GG{p} = \{ X \in \gl(N) : \D\sigma(X) = -X \} 
$ 
is the classical set of symmetric matrices. The latter is not a 
subalgebra of $\gl(N)$ but is closed under $\CC{T}_{X}$, hence is a Lie triple system.
The decomposition \refp{eq:algebrasplitting} is nothing else than the
canonical decomposition of a matrix into its skew-symmetric and
symmetric part,
$ 
  X = P+K = \frac12(X-\D\sigma(X)) + \frac12(X+\D\sigma(X)) = 
  \frac12(X-X^{\CC{T}}) + \frac12(X+X^{\CC{T}}).
$ 
It is well known that the polar decomposition $x = sq$ can be characterized in
terms of best approximation properties. The orthogonal part $q$ is the
best orthogonal approximation of $x$ in any orthogonally invariant
norm
(e.g.\ 2-norm and Frobenius norm).
Other classical polar decompositions $x=sq$, with $s$ Hermitian and 
$q$ unitary, or with $s$ real and $q$ coinvolutory (i.e.\ $q\bar q=I$), 
can also be fitted in this framework \cite{debruijn55ose} with the 
choice of automorphisms $\sigma(x) = x^{-*}= \bar x^{-\CC{T}}$ 
(Hermitian adjoint), and $\sigma(x) = \bar{x}$ respectively (where $\bar x$ denotes the complex conjugate of $x$). 

The group decomposition $x=sq$ can also be studied via the algebra
decomposition.

\subsection{Generalized polar decompositions}
In~\cite{munthe-kaas01gpd} such decompositions are generalized to
arbitrary involutive automorphisms, and best approximation properties
are established for the general case.

In~\cite{zanna00rrf} an explicit recurrence is given, if $\exp(X) =
x$, $\exp(S)=s$ and $\exp(Q)=q$ then $S$ and $Q$ can be expressed in
terms of commutators of $P$ and $K$.
 The first terms in the expansions of $S$ and $Q$ are
\begin{equation}
    \begin{meqn}
    S &=& P - \Frac12 [P,K]  - \Frac16[K,[P,K]] \\
    && \mbox{} + \Frac1{24}[P,[P,[P,K]]] - 
    \Frac1{24}[K,[K,[P,K]]] 
    \\
    && \mbox{} + [K,[P,[P,[P,K]]]] - \Frac1{120}[K,[K,[K,[P,K]]]] 
      - \Frac1{180}[[P,K],[P,[P,K]]] + \cdots,\\
    Q &=&  K - \Frac1{12}[P,[P,K]] + \Frac1{120}[P,[P,[P,[P,K]]]]\\
    && \qquad \mbox{}   + \Frac1{720}[K,[K,[P,[P,K]]]] 
      - \Frac1{240}[[P,K],[K,[P,K]]] + \cdots.
    \end{meqn}
    \label{eq:theory.6}
\end{equation}

Clearly, also other types of automorphisms can be considered, generalizing the group factorization \refp{eq:polardecomp} and the algebra splitting \refp{eq:algebrasplitting}. 
For instance, there is a large source of involutive automorphisms in the set of involutive \emph{inner automorphisms} 
\begin{equation}
\sigma(x) = rxr, \qquad \D\sigma(X) = r X r,
\label{eq:inner_auto}
\end{equation}
that can be applied to subgroups of $G=GL(n)$ to obtain a number of interesting factorizations. The matrix $r$ has to be involutive, $r^2=I$, but it need not be in the group: the factorization makes sense as long as $\sigma(x)$ is in the group (resp.\ $\D\sigma(X)$ is in the algebra). As an example, let $G=SO(n)$ be the group of orthogonal matrices, and let $r=[-e_1, e_2, \ldots, e_n] = I - 2 e_1 e_1^T$, where 
$e_i$ denotes the $i$th canonical unit vector in $\RR^n$. Obviously, $r \not \in SO(n)$, as $\det r = -1$; nevertheless, we have $(rxr)^T(rxr) = r^T x^T r^T r x r = I$, as long as $x \in SO(n)$, since $r=r^T$ and $r^Tr=I$, thus $\sigma(x) = rxr \in SO(n)$.
It is straightforward to verify that the subgroup $G^\sigma$ of
Theorem~\ref{th:2.1} consists of all orthogonal $n\times n$
matrices of the form
\[ q = \left(\begin{array}{cc}
      1 & 0^T \\
      0 & q_{n-1}\end{array}\right),\]
    where $q_{n-1}\in SO(n-1)$.
    Thus the corresponding symmetric space is $G/G^\sigma =
      SO(n)/SO(n-1)$. Matrices belong to the same coset if their
      first column coincide, thus the symmetric space can be
      identified with the $(n-1)$-sphere $S^{n-1}$.

      The corresponding splitting of a skew-symmetric matrix $V\in \g = \so(n)$ is
      \[
      V = \left(\begin{array}{cc}
      0 & -v^T \\
      v & V_{n-1}\end{array}\right) =
    \left(\begin{array}{cc}
      0 & -v^T \\
      v & 0\end{array}\right) + \left(\begin{array}{cc}
      0 & 0 \\
      0 & V_{n-1}\end{array}\right) = P+K\in \GG{p} \oplus \GG{k}.
      \]
     Thus any orthogonal matrix can be expressed as the product of the
     exponential of a matrix in $\GG{p}$ and one in $\GG{k}$.
     The space $\GG{p}$ can be identified with the tangent space to
     the sphere in the point $(1,0,\ldots,0)^T$. Different choices of $r$ give different interesting algebra splittings and corresponding group factorizations. For instance, by choosing $r$ to be the anti-identity, $r = [e_n, e_{n-1}, \ldots, e_1]$, one obtains an algebra splitting in \emph{persymmetric} and \emph{perskew-symmetric} matrices. The choice $r=[-e_1, e_2, \ldots, e_n, -e_{n+1}, e_{n+2} \ldots, e_{2n}]$ for symplectic matrices, gives the splitting in lower-dimensional symplectic matrices forming a sub-algebra and a Lie-triple system, and so on.
      In~\cite{zanna01gpd,iserles05eco}, such splittings are used for the efficient approximation of the exponential of skew-symmetric, symplectic and zero-trace matrices. In~\cite{krogstad01alc} similar
     ideas are used to construct computationally effective numerical integrators for
     differential equations on Stiefel and Grassman manifolds.

\subsection{Generalized polar coordinates on Lie groups}
A similar framework can be used to obtain coordinates on Lie groups \cite{krogstad03gpc}, which are of interest when solving differential equations of the type 
\begin{displaymath}
	\dot x = X(t,x) x, \qquad x(0) = e,
\end{displaymath}
by reducing the problem recursively to spaces of smaller dimension. Recall that $x(t) = \exp(\Omega(t))$ where $\Omega$ obeys the differential equation $\dot{\Omega} = \dexp^{-1}_\Omega X$, where $\dexp^{-1}_A= \sum_{j=0}^\infty \frac{B_j}{j!} \ad_A^j$, $B_j$ being the $j$th Bernoulli number, see \cite{iserles00lgm}.

Decomposing $ X = \Pi_\sigma^{-}X + \Pi_\sigma^{+}X \in \mathfrak{p}\oplus \mathfrak{k}$, where $\Pi_\sigma^{\pm}X = \frac12 (X \pm \D \sigma (X))$,  the solution can be  factorized as
\begin{displaymath}
	x(t) = \exp(P(t)) \exp(K(t)),
\end{displaymath}
where
\begin{eqnarray}
	\dot P&=& \Pi^{-}X - [P, \Pi^{+}X] + \sum_{j=1}^\infty 2^{2j} c_{2j} \mathrm{ad}^{2j}_P  \Pi^{-}X, \quad c_{2j}=\frac{B_{2j}}{(2j)!},
	\label{eq:pdot} \\
	\dot K &=& \mathrm{dexp}^{-1}_K (\Pi^{+}X - 2 \sum_{j=1}^\infty (2^{2j}-1) c_{2j} \mathrm{ad}^{2j-1}_P \Pi^{-}X).	
	\label{eq:kdot}
\end{eqnarray}
Note that \refp{eq:kdot} depends on $P$, however, it is possible to formulate it solely in terms of $K$, $\Pi^{-}X$ and $\Pi^{+}X$, but the expression becomes less neat.
In block form:
\begin{displaymath}
	\left( \begin{array}{c} \dot P \\ \dot K \end{array}\right) = \left( \begin{array}{cc} I & 0  \\0 & \mathrm{dexp}^{-1}_{K} \end{array}\right) 
	\left( \begin{array}{cc} \frac{v\cosh v}{\sinh v} & -v\\ \frac{1-\cosh v}{\sinh v} &1 \end{array}\right) \left( \begin{array}{c} \Pi^{-}Z \\ \Pi^{+} Z \end{array}\right) 
\end{displaymath}
where $v = \mathrm{ad}_P$. The above formula paves the road for a recursive decomposition, by recognizing that $\mathrm{dexp}^{-1}_K $ is the $\mathrm{dexp}^{-1}$ function on the restricted sub-algebra $\mathfrak{k}$. By introducing a sequence of involutive automorphisms $\sigma_i$, one induces a sequence of subalgebras, $\mathfrak{g} =\mathfrak{g}_0 \supset \mathfrak{g}_1 \supset \mathfrak{g}_2 \supset \ldots$,  of decreasing dimention,  $\mathfrak{g}_{i+1} = \mathrm{Range} (\Pi_{\sigma_i}^+)$.   Note also that the functions appearing in the above formulation are all analytic functions of the $\mathrm{ad}$-operator, and are either odd or even functions, therefore they can be expressed as functions of $\mathrm{ad}^2$ on $\mathfrak{p}$. In particular, this means that, as long as we can compute analytic functions of the $\mathrm{ad}^2$ operator, the above decomposition is computable. 
  
      Thus, the problem is reduced to the computation of analytic functions of the 2-cyclic part $P$ as well as analytic functions of  $\ad_P$ (trivialized tangent maps and their inverse). The following theorem addresses the computation of such functions using the same framework of Lemma~\ref{th:proj}.
      
     \begin{theorem}[\cite{krogstad03gpc}]
     	Let $P$ be the 2-cyclic part of $X=P+(X-P)$ with respect to the involution $S$, i.e.\ $SPS=-P$. Let $\Theta = P^2\Pi^{-}_S$. For any analytic function $\psi(s)$, we have
	\begin{equation}
		\psi(P) = \psi(0) I + \psi_1(\Theta) P + P \psi_1(\Theta) + P\psi_2(\Theta) P + \psi_2(\Theta) \Theta,
	\end{equation}
	where $\psi_1(s) = \frac{1}{2\sqrt{s}}(\psi(\sqrt{s})-\psi(-\sqrt{s}))$  and $\psi_2(s) = \frac{1}{2s} (\psi(\sqrt{s}) + \psi(-\sqrt{s}) - 2\psi(0))$.
     \end{theorem}
A similar result holds for $\ad_P$, see \cite{krogstad03gpc}. It is interesting to remark that if 
\begin{displaymath}
	P = \left( \begin{array}{cc} 0 & B^T \\A & 0
	\end{array} 
	\right), 
\end{displaymath}
then
\begin{displaymath}
	\Theta =P^2\Pi_S^{-} = \left( \begin{array}{cc}  B^T A & 0 \\ 0 &0
	\end{array} 
	\right), \qquad \psi_i(\Theta) = \left( \begin{array}{cc}  \psi_i(B^T A) & 0 \\ 0 &\psi_i(0)I
	\end{array} 
	\right),
\end{displaymath}
where $\psi_1, \psi_2$ as above. In particular, the problem is reduced to computing analytic functions of the principal square root of a matrix \cite{higham2008functions}: numerical methods that compute these quantities accurately and efficiently are very important for competitive numerical algorithms. Typically, $\Theta$ is a low-rank matrix, hence computations can be done  using eigenvalues and eigenfunctions  of the $\mathrm{ad}$ operator restricted to the appropriate space, see \cite{krogstad01alc,krogstad03gpc}.
In particular, if $A,B$ are vectors, then $B^T A$ is a scalar and the formulas become particularly simple.

These coordinates have interesting applications in control theory. Some early use of these generalized Cartan decompositions \refp{eq:2.2}--\refp{eq:2.3} (which the author calls $\mathbb{Z}_2$-\emph{grading}) to problems with nonholonomic constraints can be found in \cite{brockett99esc}. In \cite{khaneja2001toc}, the authors embrace the formalism of symmetric spaces and use orthogonal (Cartan) decompositions with applications to NMR spectroscopy and quantum computing, using adjoint orbits as main tool. Generally, these decompositions can be found in the literature, but have been applied mostly to the cases when $[\mathfrak{k}, \mathfrak{k} ]=\{0\} $ or $[\mathfrak{p}, \mathfrak{p} ]=\{0\} $, or both, as \refp{eq:pdot}-\refp{eq:kdot} become very simple and the infinite sums reduce to one or two terms.  
The main contribution of \cite{krogstad03gpc} is the derivation of such differential equations \emph{and} the evidence that such equations can be solved efficiently using linear algebra tools. See also \cite{zanna11gpd} for some applications to control theory. 

\subsection{Symmetries and reversing symmetries of differential
  equations} 
\label{sec:3.3}
Let $\CC{Diff}(M)$ be the group of diffeomorphisms of a manifold $M$ 
onto itself. We say that a map $\varphi \in \CC{Diff}(M)$ has a symmetry 
${\cal S}:M \to M$ if  
\begin{displaymath}
    {\cal S} \varphi {\cal S}^{-1} = \varphi
\end{displaymath}
(the multiplication indicating the usual composition of maps, i.e.\ 
$\varphi_{1} \varphi_{2} = \varphi_{1} \circ \varphi_{2}$), while if
\begin{displaymath}
    {\cal R} \varphi {\cal R}^{-1} = \varphi^{-1},
\end{displaymath}
we say that ${\cal R}$ is a reversing symmetry of $\varphi$ 
\cite{mclachlan98nit}. Without further ado, we restrict ourselves to involutory symmetries, the main subject of this paper.
Symmetries and reversing symmetries are very important 
in the context of dynamical systems and their numerical integration.
For instance, nongeneric bifurcations can become generic in the
presence of symmetries and vice versa. Thus, when using the integration
time-step as a bifurcation parameter, it is vitally important to remain
within the smallest possible class of systems. Reversing
symmetries, on the other hand, give rise to the existence of invariant tori and
invariant cylinders \cite{moser73sar,roberts92cat,sevryuk86rs,stuart96dsa}.

It is a classical result that the set of symmetries possess the
structure of a group --- they behave like automorphisms and fixed sets
of automorphisms. The group structure, however, does not extend to
reversing symmetries and fixed points of anti-automorphisms, and in the
last few years the set of reversing symmetries has received the attention of
numerous numerical analysts. In \cite{mclachlan98nit} it was observed that
the set of fixed points of an involutive anti-automorphism ${\cal A}_{-}$ was closed under the operation
\begin{displaymath}
  \varphi_1 \cdot \varphi_2 = \varphi_1 \varphi_2 \varphi_1 \in \fix
  {\cal A}_{-}, \quad \forall \varphi_1, \varphi_2 \in \fix {\cal A}_{-},
\end{displaymath}
that McLachlan et al.\ called ``sandwich product''.\footnote{The authors called the set of vector fields closed under the sandwich product $ \varphi_1 \cdot \varphi_2 = \varphi_1 \varphi_2 \varphi_1$ a {\em pseudogroup}.}  
Indeed, our initial goal was to understand such structures and investigate how they could be used to devise new numerical integrators for differential equations with some special geometric properties.
We recognise, cfr.\ \S\ref{sec:2.1}, that the set of fixed points of an
anti-automorphism is a symmetric space. Conversely, any connected space of invertible elements closed under the ``sandwich product'',  is the set of the fixed points of an involutive automorphism (cfr.~Theorem~\ref{th:2.1}) and has associated to it a LTS.
To show this, consider the well known symmetric BCH formula,
\begin{equation}
\begin{meqn}
	\exp(Z)&=&  \exp( X) \exp(Y) \exp(X)\\
	Z& =& 2X + Y +\frac16 T_Y(X)  -\frac16 T_X(Y) + \frac{7}{360} T_X^2(Y) -\frac{1}{360} T_Y^2 X + \cdots
	\end{meqn}
\label{eq:bch}
\end{equation}
\cite{sanz-serna94nhp}, which is used extensively in the context of \emph{splitting methods} \cite{MR2009376}. 
Because of the sandwich-type composition (symmetric space structure), the corresponding $Z$ must be in the LTS space, and this explains why it can be written as powers of the double commutator operators $T_X=\ad_X^2, T_Y=\ad_Y^2$ applied to $X$ and $Y$. A natural question to ask is: what is the automorphism $\sigma$ having such sandwich-type composition as anti-fixed points?
As $\fix \mathcal{A}_{-}=\{ z | \sigma(z)= z^{-1}\}$, we see that, by writing $z = \exp{Z}$, we have $\sigma(\exp(Z)) = (\exp(Z))^{-1} = \exp(-Z)$. In the context of numerical integrators, the automorphism $\sigma$ consists in changing the time $t$ to $-t$. This will be proved in \S\ref{sec:3.4}.

If $M$ is a finite dimensional smooth compact manifold, it is well
known that the infinite dimensional group of $\CC{Diff}(M)$ of all
smooth diffeomorphisms $M \to M$ is a Lie group, with Lie algebra
$\CC{Vect}(M)$ of all smooth vector fields on $M$, with the usual
bracket and exponential map.  It should be noted, however, that the
exponential map is not a one-to-one map, not even in the neighbourhood the 
identity element, since there exist diffeomorphisms arbitrary close 
to the identity which are not on any one-parameter subgroup and others 
which are on  many. However, the 
regions where the exponential map is not surjective become smaller and 
smaller the closer we approach the identity \cite{pressley88lg,omori70otg}, and, for our purpose, we 
can disregard these regions and assume that our results are formally 
true.

There are two different settings that we can consider in this
context. The first is to analyze the set of differentiable maps that possess a
certain symmetry (or a discrete set of symmetries). The second is
to consider the structure of the set of symmetries of a fixed
diffeomorphism. The first has a continuous-type structure while the
second is more often a discrete type symmetric space.

\begin{proposition}\label{th:3.4}
  The set of diffeomorphisms $\varphi$ that possess 
  ${\cal R}$ as an (involutive) reversing symmetry is a symmetric
  space of the type $G_\sigma$. 
\end{proposition}
\begin{proof} 
  Denote 
  \begin{displaymath}
    \sigma(\varphi) = {\cal R} \varphi {\cal R}^{-1}.
  \end{displaymath}
  It is clear that $\sigma$ acts as an automorphism,
  \begin{displaymath}
    \sigma(\varphi_{1}\varphi_{2}) = \sigma(\varphi_{1})
    \sigma(\varphi_{2}),
  \end{displaymath}
  moreover, if ${\cal R}$ is an involution then so is also $\sigma$. 
  Note that the set of diffeomorphisms $\varphi$ that possess 
  ${\cal R}$ as a reversing symmetry is the space of symmetric 
  elements $G_{\sigma}$ defined by the automorphism $\sigma$ (cfr.\ 
  \S\ref{sec:2}). Hence the result follows from 
  Theorem~\ref{th:2.1}.
\end{proof}

\begin{proposition}\label{th:3.4b}
  The set of reversing symmetries acting on a diffeomorphism $\varphi$
  is a symmetric space with the composition ${\cal R}_1 \cdot {\cal R}_2 =
  {\cal R}_1 {\cal R}_2^{-1} {\cal R}_1$.
\end{proposition}
\begin{proof}
  If $\mathcal{R}_1$ is a symmetry of $\varphi$ then so is also ${\cal
  R}^{-1}$, since $\mathcal{R}_1^{-1} \varphi^{-1} \mathcal{R}_1 = \varphi$
  and the assertion follows by taking the inverse on both sides of the
  equality. In particular, if $\mathcal{R}_1$ is a symmetry of $\varphi$ it
  is also true that $\mathcal{R}_1^{-1}$ is a reversing symmetry of $\varphi^{-1}$.
  Next, we observe that if $\mathcal{R}_1$ and $\mathcal{R}_2$ are two
  reversing symmetries of $\varphi$ then so is also $\mathcal{R}_1 {\cal
    S}^{-1} \mathcal{R}_1$, since
  \begin{displaymath}
    \mathcal{R}_1 \mathcal{R}_2^{-1} \mathcal{R}_1 \varphi (\mathcal{R}_1 \mathcal{R}_2^{-1}
    \mathcal{R}_1)^{-1} = \mathcal{R}_1 \mathcal{R}_2^{-1} \varphi^{-1} \mathcal{R}_2
    \mathcal{R}_1^{-1} = \mathcal{R}_1 \varphi \mathcal{R}_1^{-1} = \varphi^{-1}.
  \end{displaymath}
  It follows that the composition $\mathcal{R}_1 \cdot \mathcal{R}_2 =
  \mathcal{R}_1 \mathcal{R}_2^{-1} \mathcal{R}_1$ is an internal operation on the
  set of reversing symmetries of a diffeomorphism $\varphi$.

  With the above multiplication, the conditions i)--iii) of
  Definition~\ref{def:2.1} are easily verified. This proves the assertion
  in the case when $\phi$ has a discrete set of reversing
  symmetries. 

\end{proof}

 In what follows, we assume that $\mathcal{T}\in \CC{Diff}(M)$ is
 differentiable and involutory ($\mathcal{T}^{-1}=\mathcal{T}$) and $\sigma(\varphi) = \mathcal{T} \varphi \mathcal{T}$. 

Acting on  $\varphi = \exp(tX)$, we have
\begin{displaymath}
    \D\sigma X = \frac{\D}{\D t}\Big|_{t=0} \sigma \exp(tX) = {\cal
    T}_{*} X \mathcal{T}, 
\end{displaymath}
where $\mathcal{T}_{*}$ is the tangent map of $\mathcal{T}$. The pullback is
natural with respect to the Jacobi bracket,
\begin{displaymath}
  [\mathcal{T}_* X \mathcal{T}, \mathcal{T}_* Y\mathcal{T} ] = \mathcal{T}_*
  [X, Y] \mathcal{T},
\end{displaymath}
for all vector fields $X,Y$. Hence the map $\D\sigma$ is an
involutory algebra automorphism. Let $\GG{k}_\sigma$ and
$\GG{p}$ be the eigenspaces of $\D\sigma$ in $\ghe=
\CC{diff}(M)$. Then 
\begin{displaymath}
    \ghe = \GG{k} \oplus \GG{p},
\end{displaymath}
where 
\begin{displaymath}
    \GG{k} = \{ X : \mathcal{T}_{*} X = X \mathcal{T} \}
\end{displaymath}
is the Lie algebra of vector fields that have $\mathcal{T}$ as a 
symmetry and 
\begin{displaymath}
    \GG{p} = \{ X : \mathcal{T}_{*} X = - X \mathcal{T} \}
\end{displaymath} 
is the Lie triple system, vector fields corresponding to maps 
that have $\mathcal{T}$ as a reversing symmetry. 
Thus, as is the case for matrices, every vector field $X$ can be split
into two parts, 
\begin{displaymath}
  X = \frac12\Big(X+\D\sigma(X)\Big) + \frac12\Big(X-\D\sigma(X)\Big) = 
  \frac12\Big(X+ \mathcal{T}_* X \mathcal{T}\Big) + \frac12\Big(X - \mathcal{T}_* X
  \mathcal{T}\Big),
\end{displaymath}
having $\mathcal{T}$ as a symmetry and reversing symmetry respectively.

In the context of ordinary differential equation, let us consider
\begin{equation}
\label{eq:3.3}
    \frac{\D y}{\D t} = F(y), \qquad y \in \RR^{N}.
\end{equation}
Given an arbitrary involutive function $\mathcal{T}$, the vector field
$F$ can always be canonically split into two components, having ${\cal
  R}$ as a symmetry and reversing symmetry respectively. 
However, if one of these components equals zero, then the system
\refp{eq:3.3} has $\mathcal{T}$ as a symmetry or a reversing symmetry.

\subsection{Selfadjoint numerical schemes as a symmetric space}
\label{sec:3.4}
Let us consider the ODE \refp{eq:3.3}, whose exact flow will be denoted
as $\varphi = \exp(tF)$. Backward error analysis for ODEs implies that
a (consistent) numerical method for the integration of \refp{eq:3.3}  can
be interpreted as the sampling at $t=h$ of the flow $\varphi_h(t)$ of a
vector field $F_h$ (the so called \emph{modified vector field}) which is close to $F$, 
\begin{displaymath}
  \varphi_h(t) = \exp(t F_h), \quad F_h = F + h^p E_p + h^{p+1} E_{p+1} + \cdots,
\end{displaymath}
where $p$ is the order of the method (note that setting $t=h$, the
local truncation error is of order $h^{p+1}$).  

Consider next the map $\sigma$ on the set of flows depending on the
parameter $h$ defined as
\begin{equation}
  \label{eq:3.4}
  \sigma(\varphi_h(t)) = \varphi_{-h}(-t),
\end{equation}
where $\varphi_{-h}(t) = \exp(t F_{-h})$, with $F_{-h} = F +(- h)^p E_p +
(-h)^{p+1} E_{p+1} + \cdots$. 

The map $\sigma$ is involutive, since $\sigma^2 = \id$, and it is
easily verified by means of the BCH formula that $\sigma(\varphi_{1,h}
\varphi_{2,h}) = \sigma(\varphi_{1,h}) \sigma(\varphi_{2,h})$, hence
$\sigma$ is an automorphism. Consider next
\begin{displaymath}
  G_\sigma = \{ \varphi_h : \sigma(\varphi_h) = \varphi_h^{-1} \}.
\end{displaymath}
Then $\varphi_h \in G_\sigma$ if and only if $\varphi_{-h}(-t) =
\varphi_h^{-1}(t)$, namely the method $\varphi_h$ is {\em selfadjoint}.
\begin{proposition}
	The set of one-parameter, consistent, selfadjoint numerical schemes is a symmetric space in the sense of Theorem~\ref{th:2.1}, generated by $\sigma$ as in \refp{eq:3.4}.
\end{proposition}

Next, we perform the decomposition \refp{eq:2.2}. We deduce from
\refp{eq:3.4} that
\begin{displaymath}
  \D\sigma (F_h) = \frac{\D}{\D t} \Big|_{t=0} \sigma(\exp(tF_h)) = -(F
  +(-h)^p E_p + (-h)^{p+1} E_{p+1}) + \cdots = -F_{-h},
\end{displaymath}
hence,
\begin{displaymath}
  \GG{k} = \{ F_h : \D\sigma(F_h) = F_h \} = \{ F_h : - F_{-h} =
  F_h \}, 
\end{displaymath}
is the subalgebra of vector fields that are odd in $h$, and 
\begin{displaymath}
  \GG{p} = \{ F_h : \D\sigma(F_h) = - F_h \} = \{ F_h : F_{-h} =
  F_h \}, 
\end{displaymath}
is the LTS of vector fields that possess only even powers of $h$.
Thus, if $F_h$ is the modified vector field of a numerical integrator
$\varphi_h$, its canonical decomposition in $\GG{k} \oplus
\GG{p}$ is
\begin{eqnarray*}
  F_h &=& \frac12 (F_h + \D\sigma(F_h)) + \frac12(F_h - \D\sigma(F_h))
  \\
  &= & \frac12(F + \sum_{k=p}^\infty (1-(-1)^k) h^k E_k) + \frac12(F +
  \sum_{k=p}^\infty (1+(-1)^k) h^k E_k),
\end{eqnarray*}
the first term containing only odd powers of $h$ and the second only
even powers.
Then, if the numerical method $\varphi_h(h)$ is selfadjoint, it
contains only odd powers of $h$ locally (in perfect agreement with
classical results on selfadjoint methods \cite{hairer87sod}). 

\subsection{Connections with the generalized Scovel projection for differential equations with reversing symmetries}
\label{sec:3.5}
In \cite{munthe-kaas01gpd} it has been shown that it is
possible to generalize the polar decomposition of matrices to Lie
groups endowed with an involutive automorphism: 
every Lie group element $z$ sufficiently close to the identity can be
decomposed as $z=xy$ where $x \in G_\sigma$, the space of symmetric
elements of $\sigma$, and $y\in G^\sigma$, the subgroup of $G$ of
elements fixed under $\sigma$. Furthermore, setting $z = \exp(tZ)$ and
$y = \exp(Y(t))$, one has that $Y(t)$ is an odd function of $t$ and it
is a best approximant to $z$ in $G^\sigma$ in $G^\sigma$
right-invariant norms constructed by means of the Cartan--Killing
form, provided that $G$ is semisimple and that the decomposition 
$\ghe = \GG{p} \oplus \GG{k}$ is a Cartan decomposition.

Assume  that $\varphi$, the exact flow of the differential equation
\refp{eq:3.3}, has ${\cal R}$ as a reversing symmetry (i.e.\ $F \in
\GG{p}_\sigma$, where $\sigma(\varphi) = {\cal R} \varphi {\cal
  R}^{-1}$), while its approximation $\varphi_{h}$ has not. We perform
the polar decomposition 
\begin{equation}
  \label{eq:3.polar}
    \varphi_{h} = \psi_{h} \chi_{h}, \qquad \sigma(\psi_h) = \psi_h^{-1}, \quad \sigma(\phi_h) = \chi_h,
\end{equation}
i.e.\ $\psi_h$ has $\mathcal{R}$ as a reversing symmetry, while $\chi_{h}$ has ${\cal R}$ as a symmetry. 
Since the original flow has $\mathcal{R}$ as a reversing symmetry (and not symmetry), $\chi_h$ is the factor that we wish to eliminate.
We have $\psi_{h}^{2} = 
\varphi_{h} \sigma(\varphi_{h})^{-1}$. Hence the method obtained 
composing $\varphi_{h}$ with $\sigma(\varphi_{h})^{-1}$ has the 
reversing symmetry ${\cal R}$ every other step. To obtain $\psi_h$ we
need to extract the square root of the flow $\varphi_{h}
\sigma(\varphi_{h})^{-1}$. Now, if $\phi(t)$ is a flow, then its square
root is simply $\phi(t/2)$. However, if $\phi_h(t)$ is the flow of a
consistent numerical method ($p \geq 1$), namely the numerical
integrator corresponds to $\phi_h(h)$, it is not possible to evaluate
the square root $\phi_h(h/2)$ by simple means as it is not the same as the
numerical method with half the stepsize, $\phi_{h/2}(h/2)$. The
latter, however, offers an approximation to the square root: note that
\begin{displaymath}
  \phi_{\frac{h}2}\left(\frac{h}2\right)
  \phi_{\frac{h}2}\left(\frac{h}2\right) = \exp\left(hF +
  h(\frac{h}2)^p E_p\right) + \cdots,
\end{displaymath}
an expansion which, compared with $\phi_h(h)$, reveals that the error
in approximating the square root with the numerical method with half
the stepsize is of the order of
\begin{displaymath}
  \left( \frac{2^p-1}{2^p} \right) h^{p+1} E_p,
\end{displaymath}
a term that is subsumed in the local truncation error.
The choice
\begin{equation}
	\tilde \psi_{h}= \varphi_{h/2}\sigma(\varphi_{h/2})^{-1}=  \varphi_{h/2}\sigma(\varphi_{h/2}^{-1})
	\label{eq:approx_psi}
\end{equation}
 as an approximation to $\psi_h$ (we stress
that each flow is now evaluated at $t=h/2$),  yields a map that has the reversing  symmetry ${\cal
  R}$ at each step by design, since 
\begin{displaymath}
    \sigma(\tilde \psi_{h}) = \sigma(\varphi_{h/2} 
    \sigma(\varphi_{h/2}^{-1})) = \sigma(\varphi_{h/2}) 
    \varphi_{h/2}^{-1} = \tilde \psi_{h}^{-1}.
\end{displaymath}

Note that $\tilde \psi_{h} = \varphi_{h/2} \sigma(\varphi_{h/2}^{-1})$, where 
$\varphi_{h/2}^{-1}(t) = \varphi_{-h/2}^{*}(-t)$ is the inverse (or adjoint) method of   
$\varphi_{h/2}$. If $\sigma$ is given by \refp{eq:3.4}, then
$\sigma(\varphi_{h/2}^{-1}) = \varphi_{h/2}^*(h/2)$  and
this  algorithm is precisely the \emph{Scovel projection}
\cite{scovel91sni} originally proposed to 
to generate selfadjoint numerical schemes from an arbitrary integrator, and then generalized to the context of reversing symmetries \cite{mclachlan98nit}.

\begin{proposition}\label{th:3.5}
  The generalized Scovel projection is equivalent to choosing the
  $G_{\sigma}$-factor in the polar decomposition of a flow $\varphi_h$
  under the involutive automorphism $\sigma(\varphi) = {\cal R}
  \varphi {\cal R}^{-1}$, whereby square roots of flows are
  approximated by numerical methods with half the stepsize.
\end{proposition}

\subsection{Connection with the Thue--Morse sequence and  differential equation with symmetries}
Another algorithm that can be related to the generalized polar
decomposition of flows is the application of the Thue--Morse sequence
to improve the preservation of symmetries by means of a numerical
integrator \cite{iserles99aps}. Given an involutive automorphism $\cal
\sigma$ and a numerical method $\varphi_h$ in a group $G$ of numerical
integrators, Iserles et al.\ \cite{iserles99aps} construct the sequence of
methods 
\begin{equation}
  \varphi^{[0]} := \varphi_h, \quad \varphi^{[k+1]} := \varphi^{[k]}
  \sigma (\varphi^{[k]}), \qquad k=0,1, 2,\ldots.
  \label{eq:thue}
\end{equation}
Since $\varphi^{[k]} = \sigma^0 \varphi^{[k]}$, it is easily
observed that the $k$-th method corresponds to composing $\sigma^0 \varphi^{[k]}$ and $\sigma^1 \varphi^{[k]}$ according to the $k$-th Thue--Morse sequence
$01101001\ldots$, as displayed below in Table~\ref{tab:3.1} (see \cite{thue77smp,morse21rgo}).
\begin{table}[h]
  \begin{center}
    \leavevmode
    \begin{tabular}{|c|l|l|}
      \hline 
      $k$ & $\varphi^{[k]}$ & sequence\\ \hline
      0   &  $\sigma^0(\varphi)$       & `0'\\
      1   &  $\sigma^0( \varphi) \sigma^1( \varphi)$ & `01' \\
      2   &  $\sigma^0(\varphi)\sigma^1( \varphi) \sigma^1(\varphi) \sigma^0(\varphi)$ &
      `0110'\\
      3   &  $\sigma^0(\varphi)\sigma^1( \varphi ) \sigma^1(\varphi) \sigma^0(\varphi) \sigma^1(\varphi)\sigma^0(\varphi) \sigma^0(\varphi) \sigma^1(\varphi)$ & `01101001'\\
      \hline
    \end{tabular}
    \caption{Thue--Morse iterations for the method $\varphi^{[k]}$.}
    \label{tab:3.1}
  \end{center}
\end{table}
Iserles et al.\ (\cite{iserles99aps}) showed, by a recursive use of the BCH formula, that each iteration improves by one
order the preservation of the symmetry ${\cal S}$ by a consistent numerical method, where ${\cal S}$ is
the involutive automorphism such that $\sigma (\phi) = {\cal S} \phi
{\cal S}^{-1}$. The main argument of the proof is that if the method $\phi^{[k]}$ has a given symmetry error, the symmetry error of $\sigma(\phi^{[k]})$ has the opposite sign. Hence the two leading symmetry errors cancel and $\phi^{[k+1]}$ has a symmetry error of one order higher. In other words, if the method $\varphi_h$ preserves ${\cal
  S}$ to order $p$, then $\varphi^{[k]}$ preserves the symmetry ${\cal
  S}$ to order $p+k$ every $2^k$ steps.  As $\sigma$ changes the sign of the symmetry error only, if a method $\varphi_h$ has a given approximation order $p$, so does $\sigma (\varphi_h)$. Thus $\phi^{[0]} = \sigma (\varphi_h)$ can be used as initial condition in \refp{eq:thue}, obtaining the conjugate Thue--Morse sequence $10010110\ldots$. By a similar argument, also 
   \begin{equation}
  \chi^{[0]} := \varphi_h \hbox{ or } \chi^{[0]} :=\sigma(\varphi_h), \quad \chi^{[k+1]} := \sigma (\chi^{[k]}) \chi^{[k]}, \qquad k=0,1, 2,\ldots,
  \label{eq:thue_compl}
\end{equation}
generate sequences with increasing order of preservation of symmetry.
 
 \begin{example}
 As an illustration of the technique, consider the spatial PDE  $u_t = u_x+u_y + f(u)$, where $f$ is analytic in $u$. Let the PDE be defined over the domain $[-L,L]\times[-L,L]$ with periodic boundary conditions and initial value $u_0$. If $u_0$ is symmetric on the domain, i.e.\ $u_0(y,x) = u_0(x,y)$, so is the solution for all $t$. The symmetry is $\sigma(u(x,y)) = u(y,x)$. Now, assume that we solve the equation by the method of alternating directions, where the method $\varphi$ corresponds to solving with respect to the $x$ variable keeping $y$ fixed, while $\sigma(\varphi)$ with respect to the $y$ variable keeping $x$ fixed. The symmetry will typically be broken at the first step. Nevertheless, we can get a much more symmetric solution if the sequence of the directions obeys the Thue--Morse sequence (iteration $\varphi^{[k]}$) or the equivalent sequence given by iteration $\chi^{[k]}$. This example is illustrated in Figures~\ref{fig:alt_dir}-\ref{fig:alt_dir2}. 
 
 \begin{figure}\centering
 \includegraphics*[width=\textwidth]{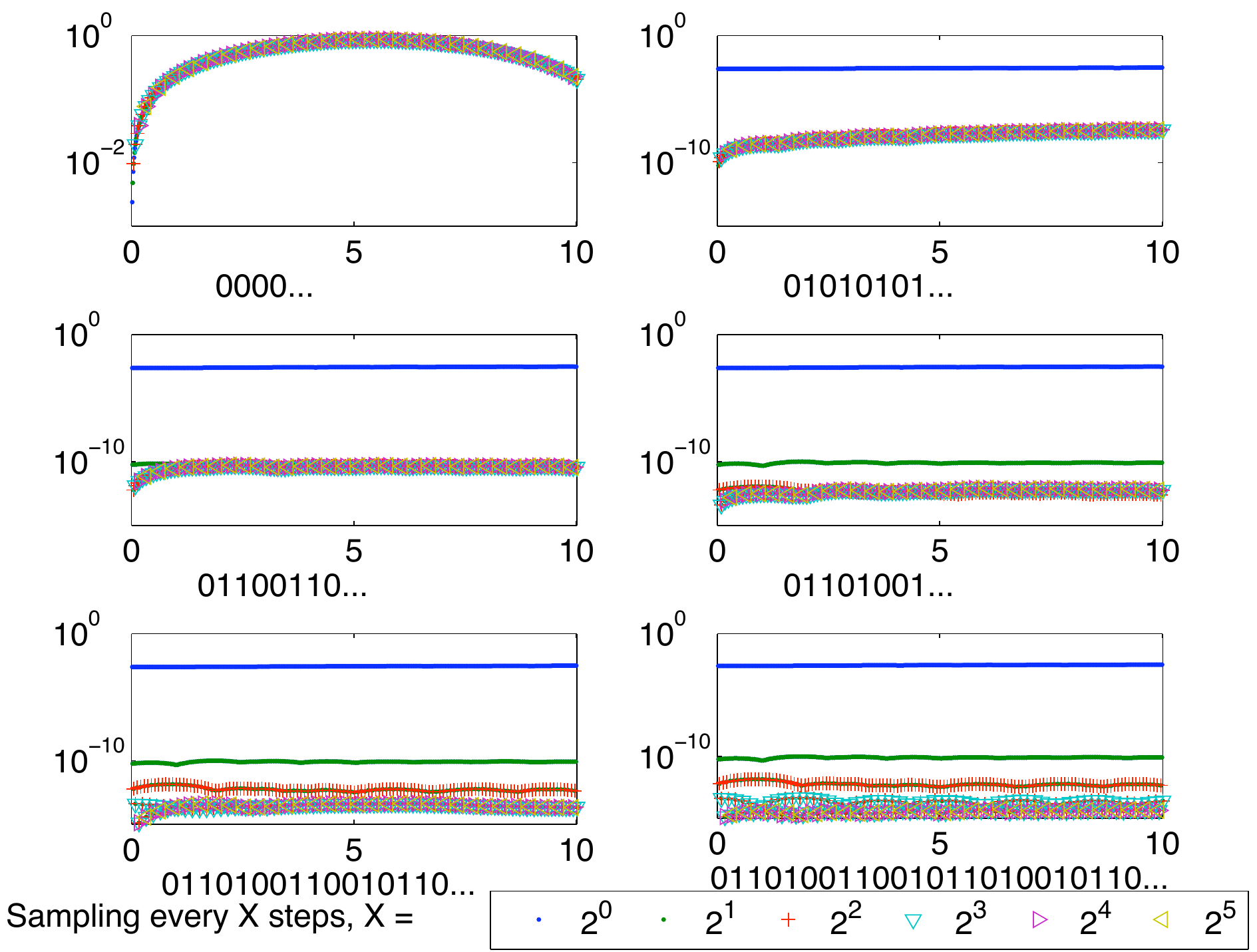}
 \caption{``Modes'' of the symmetry error $\max_{x,y} |u((t,x,y) - u(t,y,x)|$ for the method of alternating directions for the PDE $u_t = u_x+u_y + 2\times10^{-3} u^2$, on a square, with initial condition $u_0 = e^{-x^2-y^2}$ and periodic boundary conditions. The integration is performed using the first order method $u_{k+1} = u_k + h (D u_k + \frac12 f(u_k))$ (Forward Euler), where $D$ is a circulant divided difference discretization matrix of the differential operators $\partial_x, \partial_y$ (second order central differences). The experiments are performed with constant stepsize $h=10^{-2}$.  The order of the overall approximation is the same as the original method (first order only), the only difference is the order of the directions, chosen according to the Thue--Morse sequence. The different symbols correspond to different sampling rates: every step, every second step, fourth, eight, \ldots. From top left to  bottom right:  sequences '0', '01', '0110', '01101001', etc. }
 \label{fig:alt_dir}
 \end{figure}
 
 \begin{figure}[h]
 \centering
 \includegraphics*[width=.6\textwidth]{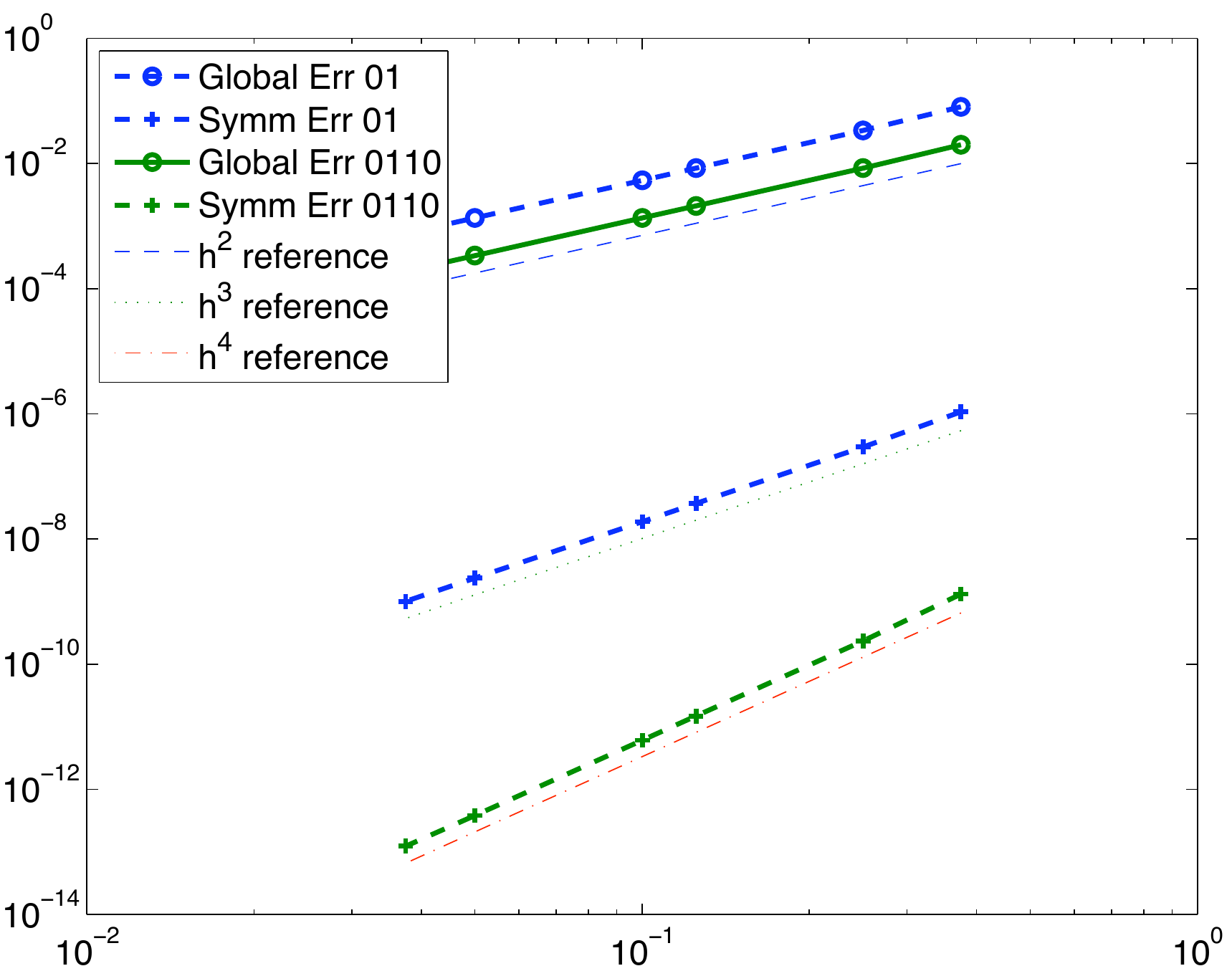}
 \caption{Global error  and symmetry error $\max_{x,y} |u((t,x,y) - u(t,y,x)|$ versus step size $h$ for the method of alternating directions for the PDE $u_t = u_x+u_y +2\times10^{-3} u^2$, as above. The basic method $\varphi$ (method ``0") is now a symmetric composition with the Heun method: half step in the $x$ direction, full step in the $y$ direction, half step in the $x$ direction, all of them performed with the Heun method (improved Euler). It is immediately observed that the symmetry error is improved by choosing the directions according to the Thue--Morse sequence. The order of the method (global error) remains unchanged (order two).
  }
 \label{fig:alt_dir2}
 \end{figure}
\end{example}

\subsection{Connections with a Yoshida-type composition and differential equations with
  symmetries}
\label{sec:3.6}
In a famous paper that appeared in 1990 (\cite{yoshida90coh}) Yoshida showed
how to construct high order time-symmetric integrators starting from lower
order time-symmetric symplectic ones. Yoshida showed that, if 
$\varphi$ is a selfadjoint numerical integrator of order $2p$, then 
\begin{displaymath}
    \varphi_{\alpha h} (\alpha t) 
    \varphi_{\beta h}(\beta t) \varphi_{\alpha h} (\alpha t) 
\end{displaymath}
is a selfadjoint numerical method of order $2p+2$ provided that the 
coefficients $\alpha$ and $\beta$ satisfy the condition
\begin{eqnarray*}
    2\alpha + \beta &=& 1 \\
    2\alpha^{2p+1} + \beta^{2p+1} &=& 0, 
\end{eqnarray*}
whose only real solution is
\begin{equation}
    \label{eq:yoshida.1}
    \alpha = \frac{1}{2-2^{1/(2p+1)}}, \qquad \beta = 
    1-2\alpha.
\end{equation}
In the formalism of this paper, time-symmetric methods correspond to 
$G_\sigma$-type elements with $\sigma$ as in \refp{eq:3.4} and it is 
clearly seen that the Yoshida technique can be used in general to 
improve the order of approximation of $G_{\sigma}$-type elements. 

A similar procedure can be applied to improve the order of the retention of symmetries
and not just reversing symmetries. To be more specific, let ${\cal 
S}$ be a symmetry of the given differential equation, namely ${\cal 
S}_{*} F = F {\cal S}$, with ${\cal S} \not = \id$, ${\cal S}^{-1} = 
{\cal S}$, ${\cal S}_{*}$ denoting the pullback of ${\cal S}$ to 
$\ghe = \CC{Vect}(M)$ (see \S\ref{sec:3.3}).  Here, the involutive 
automorphism is given by
\begin{displaymath}
    \sigma \varphi_{h}(t) = {\cal S} \varphi_{h}(t) {\cal S},
\end{displaymath}
so that
\begin{displaymath}
    \GG{p} = \{ P : {\cal S}_{*} P = - P{\cal S} \}, \quad
    \GG{k} = \{ K : {\cal S}_{*} K = K{\cal S} \}.
\end{displaymath}

\begin{proposition}\label{prop:yoshida_self}
Assume that $\varphi_{h}(t)$ is the flow of a self-adjoint numerical 
method of order $2p$, $ \varphi_{h}(t) = \exp(tF_{h})$, $    F_{h} = F + h^{2p}E_{2p} + h^{2p+2}E_{2p+2} + \cdots$,
where $E_{j} = P_{j} + K_{j}$, $P_{j} \in \GG{p}, K_{j} \in \GG{k}$ and $F$ has $\mathcal{S}$ as a symmetry. The composition
\begin{equation}
    \varphi^{[1]}_{h}(t) = \varphi_{ah}(at) \sigma(\varphi_{bh}(bt)) 
    \varphi_{ah}(at),
    \label{eq:yoshida.2}
\end{equation}
with 
\begin{displaymath}
    a = \frac{1}{2+2^{1/(2p+1)}}, \qquad b = 1-2a,
\end{displaymath}
has symmetry error $2p+2$ at $t=h$.
\end{proposition}

\begin{proof}
Write \refp{eq:yoshida.2} as
\begin{displaymath}
    \varphi^{[1]}_{h}(t) = \exp(at F_{ah}) \exp(bt \D\sigma(F_{bh})) 
    \exp(atF_{ah}).
\end{displaymath}
Application of the symmetric BCH formula, together with the fact that 
$\D\sigma$ acts by changing the signs on the $\GG{p}$-components 
only, allows us to write the relation \refp{eq:yoshida.2} as
\begin{eqnarray}
     \varphi^{[1]}_{h}(t) &=& \exp((2a+b)t F + (2(at)(ah)^{2p} +(bt) (bh)^{2p}) K_{2p}
     \label{eq:yoshida.3} \\
    &&\qquad \mbox{} +
     (2at(ah)^{2p}-bt(bh)^{2p})P_{2p} 
     + \OO{th^{2p+2}} + \OO{t^{3}h^{2p}} + \cdots), \nonumber
\end{eqnarray}
where the $\OO{th^{2p+2}}$ comes from the $E_{2p+2}$ term and the $\OO{t^3 h^{2p}}$ from the commutation of the $F$ and $E_{2p}$ terms (recall that no first order commutator appears in the symmetric BCH formula).
The numerical method is obtained letting $t=h$. We require $2a+b=1$ for consistency, and
$ 2a^{2p+1}-b^{2p+1}=0$ to annihilate the coefficient of $P_{2p}$, the lowest order $\GG{p}$-term. 
The resulting method $\varphi^{[1]}_{h}(t)$ retains the symmetry ${\cal S}$ to order 
$2p+2$, as the first leading symmetry error is a $\OO{h^{2p+3}}$ term.
\end{proof}

This procedure allows us to gain two extra degrees in the 
retention of symmetry per iteration, provided that the underlying method is selfadjoint, compared with the Thue--Morse 
sequence of \cite{iserles99aps} that yields one extra degree in symmetry per iteration but does not require selfadjointness. 
As for the Yoshida technique, the composition \refp{eq:yoshida.2} can be iterated $k$ times to obtain a time-symmetric method of order $2p$ that retains symmetry to order $2(p+k)$.
The disadvantage of \refp{eq:yoshida.2}, with respect to the classical Yoshida approach is the fact that the order of the method is retained (and does not increase by 2 units as the symmetry error). The main advantage is that all the steps are positive, in particular the second step, $b$, whereas the $\beta$ is always negative for $p\geq 3$ and typically larger than $\alpha$, requiring step-size restrictions for stiff methods. In the limit, when $p\to \infty$, $a,b\to 1/3$, i.e.\ the step lengths become equal. Thus, the proposed technique for improving symmetry is of particular interest in the context of stiff problems.
This is illustrated by the following example and Figure~\ref{fig:stiff}. 

\begin{example}
\label{ex:1}
Consider the PDE $u_t = \nabla^2 u - u (u-1)^2$, defined on the square $[-1,1]\times [-1,1]$, with a gaussian initial condition, $u(0) = \mathrm{e}^{-9x^2-9y^2}$, and periodic boundary conditions. The problem is semi-discretized on a uniform and isotropic mesh with spacing $\delta=0.1$ and is reduced to the set of ODEs $ \dot{u} = (D_{xx} + D_{yy}) u + f(u)=F(u)$, where $D_{xx}=I \otimes D_2$ and $D_{yy}= D_2 \otimes I$, $D_2$ being the circulant matrix of the standard second order divided differences, with stencil $\frac1{\delta^2}[1,-2,1]$. For the time integration, we consider the splitting $F = F_1 + F_2$, where $F_1 = D_{xx} + \frac12 f$ and $F_2 = D_{yy} + \frac12f$ and the second order self-adjoint method:\footnote{We have simply split the nonlinear term in two equal parts. Surely, it can be treated in many different ways, but that is besides the illustrative scope of the example.}
\begin{equation}
	\varphi_h = \phi^{\mathrm{BE}}_{\frac{h}2,F_1}\circ\phi^{\mathrm{BE}}_{\frac{h}2,F_2}\circ	\phi^{\mathrm{FE}}_{\frac{h}2,F_2}\circ\phi^{\mathrm{FE}}_{\frac{h}2,F_1}.
	\label{eq:pde}
\end{equation}
 We display the global error and the symmetry error for time-integration step sizes $h =3 h_0 \times [1, \frac12, \ldots,\frac{ 1}{64}]$, where the parameter $h_0$ is chosen to be the largest step size for which the basic method $\varphi$ in \refp{eq:pde} is stable. The factor $3$ comes from the fact that both the Yoshida and our symmetrization method \refp{eq:yoshida.2} require 3 sub-steps of the basic method.  So, one step of the Yoshida and our symmetrising composition  can be expected to cost the same as the basic method \refp{eq:pde}.
 
 As we can see in Figure~\ref{fig:stiff}, the Yoshida technique, with $\alpha = 1/(2-2^{1/3})$ and $\beta=1-2\alpha$, does the job of increasing the order of accuracy of the method from two to four, and so does the symmetry error. However, since $\alpha<1/2$, the $\beta$-step is negative, and, as a consequence, it is observed that the method fails to converge for the two largest values of the step size. Conversely, our symmetrization method \refp{eq:yoshida.2} has $a = 1/(2+2^{1/3})$ and $b=1-2a$, with $b$ positive, and the method converges for all the time-integration steps.
As expected, the order is not improved, but the symmetry order is improved by two units. The symmetry $\sigma$ is applied by transposing the matrix representation $u_{i,j}$ of the solution at $(i\delta, j\delta)$, before  and after the intermediate $b$-step. Otherwise, the two implementations are identical.

\begin{figure}
\centering
\includegraphics*[width=.45\textwidth]{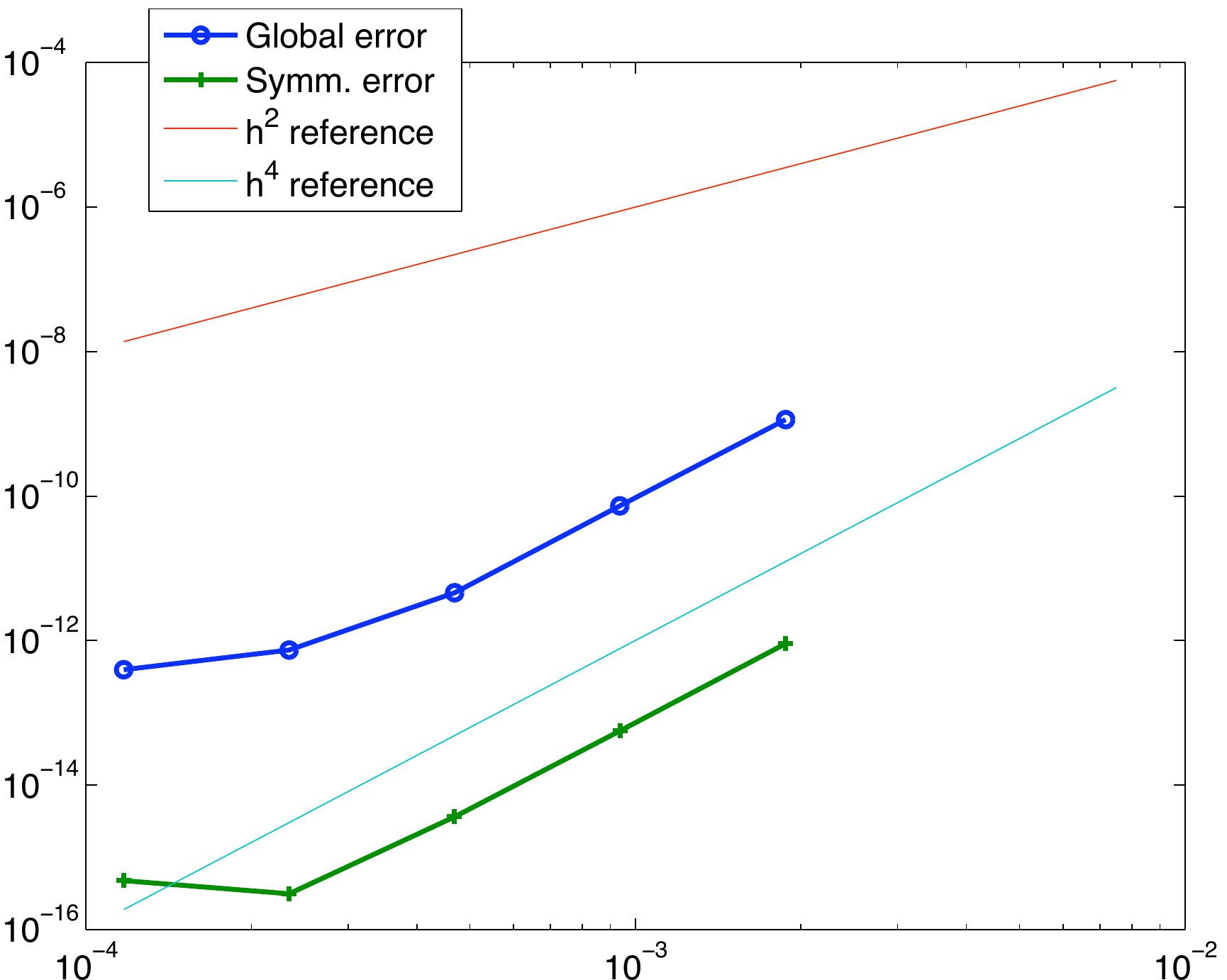}
\includegraphics*[width=.45\textwidth]{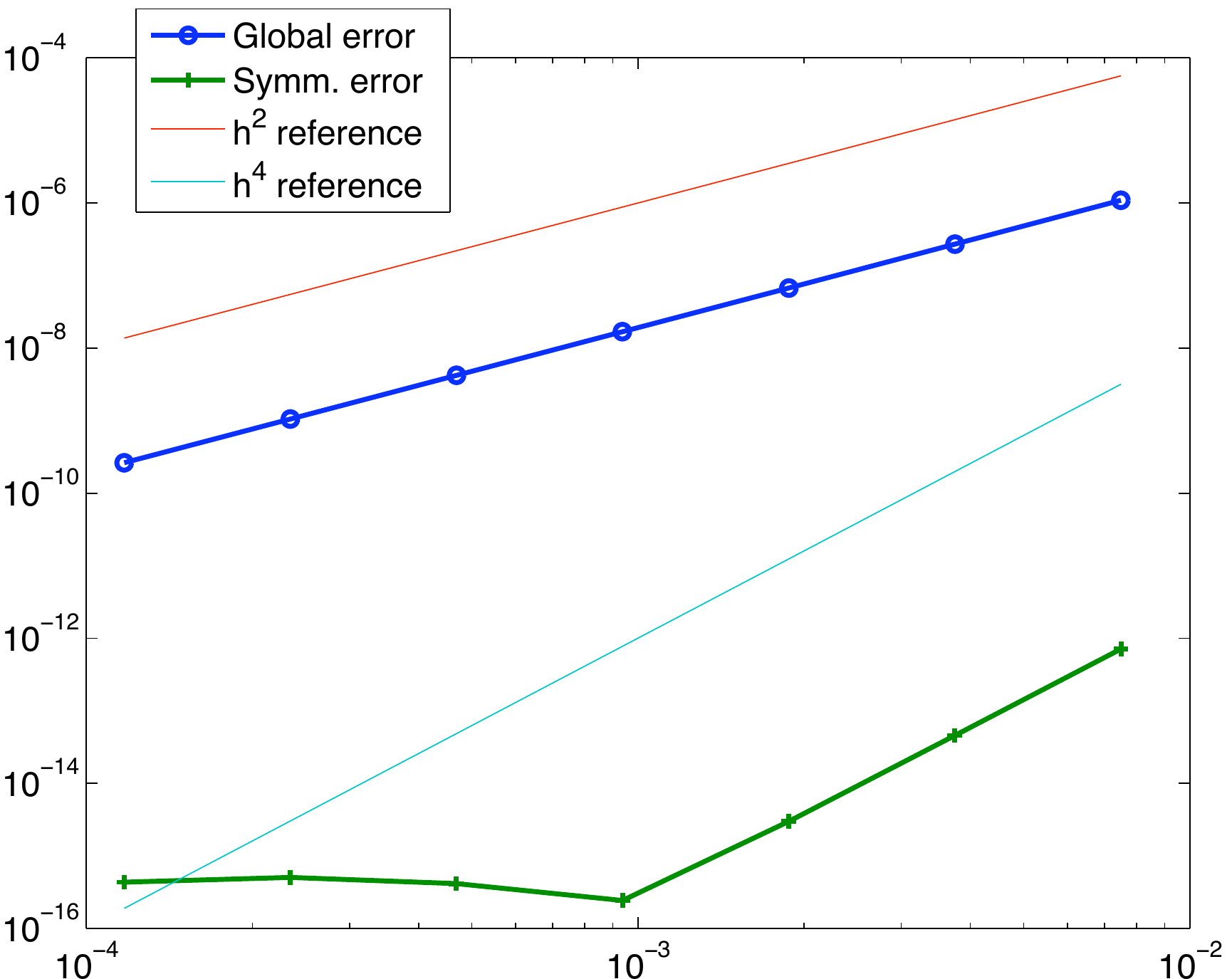}
\caption{Error versus step size for the Yoshida technique (left) and the symmetrization technique \refp{eq:yoshida.2}  (right) for the basic self-adjoint method \refp{eq:pde} applied to a semidiscretization of $u_t = \nabla^2 u - u (u-1)^2$. The order with the Yoshida techique increases by two units, but the method does not converge for the two largest values of the step size. Our symmetrization technique improves by two units the order of retention of symmetry (but not global error). Due to the positive step-sizes, it converges also for the two largest values of the step size.
See text for details.}
\label{fig:stiff}
\end{figure}
\end{example}

\section{Conclusions and remarks}
\label{sec:5}
In this paper we have shown that the algebraic structure of Lie triple systems and the factorization properties of symmetric spaces  can be used as a tool to: 1) understand and provide a unifying approach to the analysis of a number of different algorithms; and 2) devise new algorithms with special symmetry/reversing symmetry properties in the context of the numerical solution of differential equations. 
In particular, we have seen that symmetries are more difficult to retain (to date, we are not aware of methods that can retain a generic involutive symmetry in a finite number of steps), while the situation is simpler for reversing symmetries, which can be achieved in a finite number of steps using the Scovel composition. 
So far, we have considered the most generic setting where the only allowed operations are non-commutative compositions of maps (the map $\varphi$, its transformed, $\sigma(\varphi)$, and their inverses). If the underlying space is linear and so is the symmetry, i.e.\ $\sigma (\varphi_1 + \varphi_2) = \sigma(\varphi_1) + \sigma(\varphi_2)$,  the map
$	\tilde \varphi_h = \frac{\varphi_h + \sigma(\varphi_h)}2 $
obviously satisfies the symmetry $\sigma$, as $\sigma(\tilde \varphi_h) = \tilde \varphi_h$. Because of the linearity and the vector-space property, we can use the same operation as in the tangent space, namely we identify $\sigma$ and $\D \sigma$. This is in fact the most common symmetrization procedure for linear symmetries in linear spaces. For instance, in the context of the alternating-direction examples, a common way to resolve the symmetry issue is to first solve, say, the $x$ and the $y$ direction, solve the $y$ and the $x$ direction with the same initial condition, and then average the two results.

In this paper we did not mention the use and the development of similar concepts in a more strict linear algebra setting.\footnote{It seems that numerical linear algebra authors prefer to work with Jordan algebras (see \S2), rather than Lie triple systems. We believe that the LTS description is natural in the context of differential equations and vector fields because it fits very well with the Lie algebra structure of vector fields. }
Some recent  works deal with the further understanding of scalar products and structured factorizations, and more general computation of matrix functions preserving group structures, see for instance  \cite{mackey06sfi,higham05fpm,higham10tcg} and references therein. 
Some of these topics are covered by other contributions in the present BIT issue, which we strongly encourage the reader to read to get a more complete picture of the topic and its applications.

\vspace*{30pt}
\noindent
{\bf Acknowledgements.} H.M.-K., G.R.W.Q.  and A.Z. wish to thank the Norwegian
Research Council and the Australian Research Council for financial support.
Special thanks to Yuri Nikolayevsky who gave us important pointers to the literature on symmetric spaces.

\nocite{iserles00lgm}
\nocite{munthe-kaas99hor}
\bibliographystyle{plain}

\begin{thebibliography}{10}

\bibitem{brockett99esc}
R.~Brockett.
\newblock Explicitly solvable control problems with nonholonomic constraints.
\newblock In {\em Proceedings of the $38^{th}$ Conference on Decision \&
  Control, Phoenix, Arizona}, 1999.

\bibitem{debruijn55ose}
N.~G.~De Bruijn and G.~Szekeres.
\newblock On some exponential and polar representations.
\newblock {\em Nieuw Archief voor Wiskunde}, (3) III:20--32, 1955.

\bibitem{hairer87sod}
E.~Hairer, S.~P. N{\o}rsett, and G.~Wanner.
\newblock {\em Solving {O}rdinary {D}ifferential {E}quations I. Nonstiff
  Problems}.
\newblock Springer-Verlag, Berlin, 2nd revised edition, 1993.

\bibitem{helgason78dgl}
S.~Helgason.
\newblock {\em Differential {G}eometry, {L}ie {G}roups and {S}ymmetric
  {S}paces}.
\newblock Academic Press, 1978.

\bibitem{higham05fpm}
N.~J. Higham, D.~S. Mackey, N.~Mackey, and F.~Tisseur.
\newblock Functions preserving matrix groups and iterations for the matrix
  square root.
\newblock {\em SIAM J. Matrix Anal.\ and Appl.}, 26(3):849--877, 2005.

\bibitem{higham10tcg}
N.~J. Higham, C.~Mehl, and F.~Tisseur.
\newblock The canonical generalized polar decomposition.
\newblock {\em SIAM J. Matrix Anal.\ and Appl.}, 31(4):2163--2180, 2010.

\bibitem{higham2008functions}
N.J. Higham.
\newblock {\em Functions of matrices: theory and computation}.
\newblock Society for Industrial Mathematics, 2008.

\bibitem{iserles99aps}
A.~Iserles, R.~McLachlan, and A.~Zanna.
\newblock Approximately preserving symmetries in numerical integration.
\newblock {\em Euro.\ J. Appl.\ Math.}, 10:419--445, 1999.

\bibitem{iserles00lgm}
A.~Iserles, H.~Munthe-Kaas, S.~P. N{\o}rsett, and A.~Zanna.
\newblock Lie-group methods.
\newblock {\em Acta Numerica}, 9:215--365, 2000.

\bibitem{iserles05eco}
A.~Iserles and A.~Zanna.
\newblock Efficient computation of the matrix exponential by generalized polar
  decompositions.
\newblock {\em SIAM J. Numer.\ Anal.}, 42(5):2218--2256, 2005.

\bibitem{khaneja2001toc}
Navin Khaneja, Roger Brockett, and Steffen~J. Glaser.
\newblock Time optimal control in spin systems.
\newblock {\em Phys. Rev. A}, 63:032308, Feb 2001.

\bibitem{krogstad01alc}
S.~Krogstad.
\newblock A low complexity {L}ie group method on the {S}tiefel manifold.
\newblock {\em BIT}, 43(1):107--122, March 2003.

\bibitem{krogstad03gpc}
S.~Krogstad, H.~Z. Munthe-Kaas, and A.~Zanna.
\newblock Generalized polar coordinates on lie groups and numerical
  integrators.
\newblock {\em Numerische Matematik}, 114:161--187, 2009.

\bibitem{loos69sp1}
O.~Loos.
\newblock {\em Symmetric {S}paces {I}: {G}eneral {T}heory}.
\newblock W. A. Benjamin, Inc., 1969.

\bibitem{mackey06sfi}
D.S. Mackey, N.~Mackey, and F.~Tisseur.
\newblock Structured factorizations in scalar product spaces.
\newblock {\em SIAM J.\ Matrix Anal.\ Appl.}, 27(3):821--850, 2006.

\bibitem{MR2009376}
R.~I. McLachlan and G.~R.~W. Quispel.
\newblock Splitting methods.
\newblock {\em Acta Numer.}, 11:341--434, 2002.

\bibitem{mclachlan98nit}
R.~I. McLachlan, G.~R.~W. Quispel, and G.~S. Turner.
\newblock Numerical integrators that preserve symmetries and reversing
  symmetries.
\newblock {\em SIAM J. Numer.\ Anal.}, 35(2):586--599, 1998.

\bibitem{morse21rgo}
M.~Morse.
\newblock Recurrent geodesics on a surface of negative curvature.
\newblock {\em Trans.\ Amer.\ Math.\ Soc.}, 22:84--100, 1921.

\bibitem{moser73sar}
J.~Moser.
\newblock {\em Stable and Random Motion in Dynamical Systems}.
\newblock Princeton University Press, 1973.

\bibitem{munthe-kaas99hor}
H.~Munthe-Kaas.
\newblock High order {R}unge--{K}utta methods on manifolds.
\newblock {\em Applied Numerical Mathematics}, 29:115--127, 1999.

\bibitem{munthe-kaas01gpd}
H.~Munthe-Kaas, G.~R.~W. Quispel, and A.~Zanna.
\newblock Generalized polar decompositions on {L}ie groups with involutive
  automorphisms.
\newblock {\em Journal of the {F}oundations of {C}omputational {M}athematics},
  1(3):297--324, 2001.

\bibitem{omori70otg}
H.~Omori.
\newblock On the group of diffeomorphisms of a compact manifold.
\newblock {\em Proc.\ Symp.\ Pure Math.}, 15:167--183, 1970.

\bibitem{owren01imb}
B.~Owren and A.~Marthinsen.
\newblock Integration methods based on canonical coordinates of the second
  kind.
\newblock {\em Numerische Mathematik}, 87(4):763--790, Feb. 2001.

\bibitem{pressley88lg}
A.~Pressley and G.~Segal.
\newblock {\em Loop Groups}.
\newblock Oxford Mathematical Monographs. Oxford University Press, 1988.

\bibitem{roberts92cat}
J.~A.~G. Roberts and G.~R.~W. Quispel.
\newblock Chaos and time-reversal symmetry: order and chaos in reversible
  synamical systems.
\newblock {\em Phys.\ Rep.}, 216:63--177, 1992.

\bibitem{sanz-serna94nhp}
J.~M. Sanz-Serna and M.~P. Calvo.
\newblock {\em Numerical {H}amiltonian {P}roblems}.
\newblock AMMC 7. Chapman {\&} Hall, 1994.

\bibitem{scovel91sni}
J.~C. Scovel.
\newblock Symplectic numerical integration of {H}amiltonian systems.
\newblock In Tudor Ratiu, editor, {\em The Geometry of Hamiltonian Systems},
  volume~22, pages 463--496. MSRI, Springer-Verlag, New York, 1991.
\newblock Symplectic numerical integration of Hamiltonian systems.

\bibitem{sevryuk86rs}
M.~B. Sevryuk.
\newblock {\em Reversible {S}ystems}.
\newblock Number 1211 in Lect.\ Notes Math. Springer, Berlin, 1986.

\bibitem{stuart96dsa}
A.~M. Stuart and A.~R. Humphries.
\newblock {\em Dynamical Systems and Numerical Analysis}.
\newblock Cambridge University Press, Cambridge, 1996.

\bibitem{thue77smp}
A.~Thue.
\newblock {\"U}ber unendliche {Z}eichenreihen.
\newblock In T.~Nagell, editor, {\em Selected mathematical papers of Axel
  Thue}, pages 139--158. Universitetsforlaget, Oslo, 1977.

\bibitem{yoshida90coh}
H.~Yoshida.
\newblock Construction of higher order symplectic integrators.
\newblock {\em Physics Letters A}, 150:262--268, 1990.

\bibitem{zanna00rrf}
A.~Zanna.
\newblock Recurrence relation for the factors in the polar decomposition on
  {L}ie groups.
\newblock {\em Math.\ Comp.}, 73:761--776, 2004.

\bibitem{zanna11gpd}
A.~Zanna.
\newblock Generalized polar decompositions in control.
\newblock In {\em Mathematical papers in honour of F\'atima Silva Leite},
  volume~43 of {\em Textos Mat. S\'er. B}, pages 123--134. Univ. Coimbra, 2011.

\bibitem{zanna01gpd}
A.~Zanna and H.~Z. Munthe-Kaas.
\newblock Generalized polar decompositions for the approximation of the matrix
  exponential.
\newblock {\em SIAM J. Matrix Anal.}, 23(3):840--862, 2002.

\end{thebibliography}

\end{document}